\documentclass[12pt,a4paper,reqno]{amsart}
\usepackage{amsmath,amssymb,amsfonts,amsthm}
\usepackage{hyperref}
\usepackage[mathscr]{eucal}

\author{H.~M.~Khudaverdian}
\author{Th.~Th. Voronov}
\address{Department  of Mathematics, University of Manchester, Manchester, M13 9PL, UK}
\address{Institute for Information Transmission Problems of Russian Academy of Sciences (Kharkevich Institute), Moscow, Russia}
\email{khudian@manchester.ac.uk, khudian37@gmail.com}

\address{Department  of Mathematics, University of Manchester, Manchester, M13 9PL, UK}
\email{theodore.voronov@manchester.ac.uk}

\title[On  Buchstaber--Rees theory and   generalization]{On   the  Buchstaber--Rees theory of ``Frobenius $n$-homomorphisms'' and its generalization}

\date{10 (23) August 2022}

\def\co{\colon\thinspace}
\newtheorem{thm}{Theorem}[section]
\newtheorem{lm}{Lemma}[section]
\theoremstyle{definition}
\newtheorem{de}{Definition}[section]
\newtheorem{ex}{Example}[section]
\newtheorem{rem}{Remark}[section]

\def\a{\alpha}
\newcommand{\f}{\mathbf{f}}
\newcommand{\g}{\mathbf{g}}
\newcommand{\ps}{{{\psi}}}
\newcommand{\ch}{{\boldsymbol{\chi}}}
\DeclareMathOperator{\Ber}{Ber}
\DeclareMathOperator{\str}{str}
\newcommand{\ZZ}{{\mathbb Z}}
\newcommand{\RR}{\mathbb R}
\DeclareMathOperator{\fBer}{Ber_{\f}}
\newcommand{\F}{{\Phi}}
\renewcommand{\L}{{\mathcal L}}
\DeclareMathOperator{\Mat}{Mat}
\DeclareMathOperator{\diag}{diag}
\DeclareMathOperator{\Sym}{Sym}
\DeclareMathOperator{\ev}{ev}

\newcommand{\si}{{\boldsymbol{\psi}}}

\begin{document}

\maketitle

\begin{abstract}
This is a  survey of our results on the theory of  $n$-homomorphisms  of Buchstaber--Rees  and   its generalization that we obtained. In short, we are concerned with classes of linear maps between commutative rings that can be described   the    ``next level  after ring homomorphisms'' with respect to   multiplicative properties. Our main tool is a construction which we call  the ``characteristic function''  --- whose functional properties encode   algebraic properties of a linear map in question. Namely, if the characteristic function  is polynomial of degree $n$, the map is an $n$-homomorphism in the sense of Buchstaber--Rees, and our approach simplifies their theory substantially.   If the characteristic function is an irreducible  rational fraction with the numerator and denominator of degrees $p$ and $q$ respectively, we arrive at a new   notion of a ``$p|q$-homomorphism''.  Examples of $p|q$-homomorphisms are sums and differences of ring homomorphisms.  Our  construction is motivated  by our earlier results in superalgebra/supergeometry concerning Berezinians and super exterior powers.
\end{abstract}

\section{Introduction}

\subsection{}
Everybody knows that the sum of ring homomorphisms is not a homomorphism. But what is it then? Also, unlike for determinant, the trace of the product of matrices is not the product of the traces. Is there a meaningful description of how matrix trace behaves under multiplication?

The second question is classical and was investigated by Frobenius for the needs of his theory of characters of finite groups~\cite{frobenius:1896}. Frobenius introduced ``higher characters'' corresponding to a given character  (i.e., the character of a group representation in modern framework) by a recursive formula. The terms of this ``Frobenius recursion''  serve  for description of multiplicative properties of traces and characters.

\subsection{}
In 1996--2008, V.~M.~Buchstaber and E.~Rees,  with an  original motivation entirely different from   Frobenius's,  substantially advanced this circle of ideas and introduced what they called  \emph{$n$-homomorphisms}  of (commutative) algebras, or   \emph{``Frobenius $n$-homomorphisms''}. See their works~\cite{buchstaber_rees:1996,buchstaber_rees:1997,buchstaber_rees:2002,buchstaber_rees:2004,buchstaber_rees:2008}.
This notion   answers the first of the above questions, viz.: the sum of $n$ algebra homomorphisms is an $n$-homomorphism. (Informally, one may visualize all $n$-homomorphisms  this way as sums of homomorphisms.) It is interesting to mention the source of the original problem considered in~\cite{buchstaber_rees:1996}: it was a question   about an analog of a Hopf algebra structure for the algebra of functions on an ``$n$-valued group''; the latter is a notion Buchstaber introduced in early 1990s motivated by his earlier discovery of ``$n$-valued'' formal groups  arising in connection with   Pontrjagin classes in generalized cohomology theories~\cite{buchstaber:viniti1978}. So ultimately, the interest comes from topology, and the question is basically   which algebraic properties the ``diagonal'' or ``coproduct'' on the algebra of functions on an $n$-valued group has. And it is an $n$-homomorphism as established in~\cite{buchstaber_rees:1996}. (Then   the relation with  Frobenius's recursion was found~\cite{buchstaber_rees:1997}.)

According to general     Functional-Algebraic Duality   (see e.g. Shafarevich~\cite{shaf:a}), spaces correspond to rings and    maps of spaces      correspond to ring homomorphisms. In particular, points  correspond to  homomorphisms to   fields. One manifestation of this principle is   the Gelfand--Kolmogorov theorem about rings of continuous functions~\cite{gelfand-kolmogorov}. In Buchstaber and Rees's theory,    this   was extended to   $n$-homomorphisms\,---\,which were showed to correspond to maps to the  $n$th symmetric power  of a space.
\subsection{}
The authors of the present text, around 2004,    studied relation of   exterior powers and Berezinian  (superdeterminant).
The crucial difference from the usual determinant  is that Berezinian is rational and not polynomial. Moreover, the numerator and the denominator of the standard formula for Berezinian have no separate invariant meaning.  

Berezinian does not arise from  super exterior powers  defined using the sign rule  in any obvious way. There is no top power: the sequence of the exterior powers  of a superspace  stretches infinitely to the right, and  there is no place   for   superdeterminant there. (This in particular leads to various non-trivial and unexpected effects in   supermanifold integration theory, into which we do not go here, see~\cite{tv:gitnew}.)

Nevertheless,   we found that there is  a  non-obvious  connection between Berezinians and super exterior powers; in particular, it became possible to express  Berezinian of a supermatrix as the ratio   of    Hankel determinants built on the  supertraces of its exterior powers    (and which are polynomial invariants). The key to our theory was  the use of a ``rational characteristic function'' $\Ber(1+z\mathbf{A})$, where $\mathbf{A}$ is a linear operator   on a superspace. Its expansion at zero gives the supertraces of all exterior powers $\Lambda^k(\mathbf{A})$, while the expansion at infinity  gives in particular   $\Ber \mathbf{A}$. The fundamental statement here is a ``universal recurrence relation'' linking the two expansions. See~\cite{tv:ber}.

\subsection{}
It somehow came to our mind that the method that we developed for studying super exterior powers and Berezinians could be applied to the theory of  Buchstaber and Rees (about which we were  learning    first-hand), though there is no apparent connection of it with supermathematics.  Pursuing this highly informal analogy, we introduced the notion of  a   \emph{characteristic function of a linear map  of algebras} 
\begin{equation*}
  \f\co A\to B
\end{equation*}
%$\f\co A\to B$ 
and the related notion of an \emph{$\f$-Berezinian}. See~\cite{tv:frobenius-umn-eng}, also~\cite{tv:frobbieloviezha,tv:shortproof}.   The definition of this function is in the next section (Definition~\ref{def.char}). (We wish to stress to avoid a confusion that it is not about a superanalog  for the Buchstaber--Rees theory. It would be a different task, definitely worthwhile,    probably straightforward. To the contrary, we use  ideas from the super world for obtaining constructions in a classical setup.)

The   framework obtained substantially simplifies the    theory of $n$-homomorphisms.

It also  opens way  to  its  natural generalization, i.e.,  to what we have called  \emph{$p|q$-homomorphisms}. Informally,   $p|q$-homomorphisms answer  to the question about the   difference or,  more generally,        integral linear combinations of algebra homomorphisms. (Answer: such a linear combination is a $p|q$-homomorphism, where $p$ is the sum of the positive coefficients and $-q$ is   the sum of the negative coefficients.)  Our $p|q$-homomorphisms correspond to rational characteristic functions, while Buchstaber--Rees's $n$-homomorphisms correspond to polynomial characteristic functions. (One may speculate about an algebraic theory resulting from other types of functions.)

We also constructed an analog of the geometric part of the Buchstaber--Rees theory, as a generalization of symmetric powers. Here we have only part of the statements that we would like to have, and there are  questions still open.

\medskip
We dedicate this paper  to the memory of two outstanding mathematicians whom we lost in 2019:   Elmer Rees (1941--2019) and Alexandre Mikhailovich Vinogradov (1938--2019).

\section{Main tool: characteristic function for a linear map of algebras}

The setup is as follows.   Given commutative algebras with unit $A$ and $B$. Given a linear map (which is not an algebra homomorphism  in general)
\begin{equation}
 \f\co A\to B\,.
\end{equation}
Remark:  it is possible to relax the conditions on $A$ and $B$ by dropping the commutativity of $A$ at the expense of requiring that $\f$ is ``trace-like'', i.e. $\f(a_1a_2)=\f(a_2a_1)$. (This would be as in Frobenius's original problem about matrices.)   It is also possible to give a  generalization  to the supercase, but we do not do it here.

\begin{de}[\cite{tv:frobenius-umn-eng}]\label{def.char}
 The \emph{characteristic function} of $\f$, notation: $R(\f,a,z)$, is a formal power series in $z$ with coefficients in (in general, non-linear) maps $A\to B$, defined by
 \begin{equation}
  R(\f,a,z):=e^{\f\ln(1+az)}\,.
 \end{equation}
\end{de}

In greater detail,
\begin{equation}
  R(\f,a,z):=\exp\left(\sum_{k=1}^{+\infty}\frac{(-1)^{k-1}}{k}\,\f\bigl(a^k\bigr)\,z^k\right)\,.
 \end{equation}

 There is the obvious exponential property: for two linear maps $\f,\g\co A\to B$,
 \begin{equation}\label{eq.expprop}
   R(\f+\g,a,z)= R(\f,a,z) R(\g,a,z)\,.
 \end{equation}

\begin{ex}
 If $\f$ is an algebra homomorphism, then $\f\ln(1+az)=\ln(1+\f(a)z)$, hence $R(\f,a,z)=1+\f(a)z$ is a linear function of $z$.
\end{ex}

The \textbf{main idea} is to describe algebraic properties of $\f$ in terms of functional properties of $R(\f,a,z)$ as a function of $z$.

\begin{lm}
 The expansion of $R(\f,a,z)$ in $z$ is given explicitly by
 \begin{equation}
  R(\f,a,z)= 1+ \ps_1(\f,a)z+\ps_2(\f,a)z^2+\ldots
 \end{equation}
where $\ps_k(\f,a)=P_k(\f(a),\ldots,\f(a^k))$ and  $P_k(s_1,\ldots,s_k)$ are the classical Newton polynomials (which express elementary symmetric functions via sums of powers), i.e.
\begin{equation}\label{eq.psi}
 \ps_k(\f,a)=
 %P_k(s_1,\ldots,s_k)=
 \frac{1}{k!}\begin{vmatrix}
                          \f(a)    & 1          & 0      & \dots & 0 \\
                          \f(a^2)    & \f(a)        & 2      & \dots & 0 \\
                          \dots  & \dots      & \dots  & \dots &\dots \\
                          \f(a^{k-1})& \f(a^{k-2})     & \f(a^{k-3}) & \dots & k-1 \\
                          \f(a^{k})     & \f(a^{k-1})     & \f(a^{k-2}) & \dots & \f(a)
                        \end{vmatrix}\,.
\end{equation}
\end{lm}
\begin{proof}
 Classical statement about the power series $\exp\left(\sum_{k=1}^{+\infty}\frac{(-1)^{k-1}}{k}\,c_k\,z^k\right)$ applied to $c_k=\f(a^k)$\,. (See e.g. in Macdonald~\cite{macdon:symm}.)
\end{proof}

We   define the following useful notion. 

For a fixed linear map $\f\co A\to B$ as above, we say that a function $\phi\co A\to B$ is an 
\textit{$\f$-polynomial}  if its values $\phi(a)$, where $a\in A$,
are given by a universal (i.e., independent of $a$) polynomial expression in $\f(a)$, $\f(a^2)$, \ldots\,. The ring of $\f$-polynomial functions is naturally graded, so that the degree of $\f(a)$ is $1$, the degree of $\f(a^2)$ is $2$, etc. Thus, the functions  $\ps_k(\f,a)$ are $\f$-polynomials of degree $k$, for each $k=0,1,2,\ldots$

\begin{rem}
 Our construction of  $R(\f,a,z)$ was motivated by the \emph{characteristic function  of an (even)    linear operator} $\mathbf{A}\co V\to V$ acting on a superspase $V$:
 \begin{equation}
  R(\mathbf{A},z)=\Ber(E+\mathbf{A}z)\,,
 \end{equation}
%(for an even $\mathbf{A}$),
introduced in in~\cite{tv:ber}, which is a rational function of $z$ and whose expansions at $0$ and infinity were the main tools in our study of (super) exterior powers. By using Liouville's formula, we can write $\Ber(E+\mathbf{A}z)=e^{\str\ln (E+\mathbf{A}z)}$. Thus, %the characteristic function of a linear operator is a special case   corresponding to supetrace, i.e.
$R(\mathbf{A},z)=R(\str,\mathbf{A},z)$. %Generalizing from this case was our guiding principle.
\end{rem}

Originally $R(\f,a,z)$ is defined as a formal power series in $z$. Interesting consequences will follow if we assume that it is an expansion of a \textit{genuine} function  defined in the complex plane.  If we allow it to have  poles and require that it does not have  essential singularities, in particular at infinity, then $R(\f,a,z)$ will be \emph{rational}. Which we  shall assume from now on.

Clearly, by definition,
\begin{equation}
 R(\f,a,0)=1\,.
\end{equation}
What is the behaviour of $R(\f,a,z)$ at infinity?

\begin{lm}
 If the characteristic function $R(\f,a,z)$ is rational, then the element
 \begin{equation}
  \ch(\f):=\f(1)\in B
 \end{equation}
is an integer.
\begin{proof}
 Indeed,
 \begin{equation}
  R(\f,1,z)= e^{\f\ln(1+1z)}=e^{\f(1)\ln(1+z)}=(1+z)^{\f(1)}\,,
 \end{equation}
 which cannot be a rational function of $z$ unless $\f(1)$ is not in $\ZZ\subset B$. (Another argument can be found in~\cite{tv:shortproof}.)
\end{proof}

\end{lm}

\begin{lm}
 The expansion of the characteristic function $R(\f,a,z)$ near infinity (i.e. in the negative powers of $z$) has the form:
 \begin{equation}
  R(\f,a,z)= z^{\ch(\f)}e^{\f\ln(a)}\left(1+\ps_1(\f,a^{-1})z^{-1}+\ps_2(\f,a^{-1})z^{-2}+\ldots \right)
 \end{equation}
where $\ch(\f)=\f(1)\in \ZZ$. (The formula holds where the righ-hand side  makes sense.)
\end{lm}
\begin{proof}
 Formal manipulation:
 \begin{multline*}
  e^{\f\ln(1+az)}=e^{\f\ln\bigl(az(a^{-1}z^{-1}+1)\bigr)}=e^{\f(\ln(az))+\f\ln(1+a^{-1}z^{-1})}=\\
  e^{\f(\ln(az))}e^{\f\ln(1+a^{-1}z^{-1})}=e^{\f\ln(a)+\f(1)\ln(z)}e^{\f\ln(1+a^{-1}z^{-1})}=\\
  e^{\f\ln(a)}z^{\f(1)}e^{\f\ln(1+a^{-1}z^{-1})}=
 e^{\f\ln(a)}z^{\ch(\f)}R(\f,a^{-1},z^{-1})
 \end{multline*}
\end{proof}

 In particular, the order of the pole at infinity is exactly $\ch(\f)=\f(1)$.
 \begin{de}
  We define the \emph{$\f$-Berezinian} of an element $a\in A$, taking values in $B$, 
by
\begin{equation}
 \fBer(a):=e^{\f\ln(a)}
\end{equation}
whenever it makes sense.
 \end{de}
The $\f$-Berezinian of $a\in A$ is the leading coefficient of the expansion of the characteristic function  $R(\f,a,z)$ at infinity. By construction,
\begin{equation}\label{eq.multber}
 \fBer(a_1a_2)=\fBer(a_1)\fBer(a_2)\,.
\end{equation}
Indeed, $\fBer(a_1a_2)= e^{\f\ln(a_1a_2)}=e^{\f\ln(a_1)+\f\ln(a_2)}=e^{\f\ln(a_1)}e^{\f\ln(a_2)}=\fBer(a_1)\fBer(a_2)$, where we used $\f\ln(a_1a_2)=\f\ln(a_1)+\f\ln(a_2)$ (this would apply also to a non-commutative $A$, provided $\f$ is trace-like).

\begin{rem}
 It is interesting if it is possible to replace in the constructions the Riemann sphere by a   surface of a higher genus, e.g. a torus, so that ``zero'' and ``infinity'' will be just two marked points. Can a meaningful algebraic theory   arise this way, e.g. of something that can be called an ``elliptic homomorphism'' (with an elliptic characteristic function)?
\end{rem}

\section{Application to the  theory of Buchstaber--Rees}

\subsection{Definition of   ``Frobenius recursion''}

\begin{de}[Buchstaber--Rees~\cite{buchstaber_rees:1997}, following Frobenius~{\cite[Eq 22]{frobenius:1896}}]
For a linear   $\f\co A\to B$ as above, the \emph{Frobenius maps} $\F_{k}(\f,a_1,\ldots,a_{k})$ for all $k=1,2,3,\ldots$ are defined by
\begin{equation}
 \F_{1}(\f,a):=\f(a)
\end{equation}
and
\begin{multline}\label{eq.deffrob}
 \F_{k+1}(\f,a_1,\ldots,a_{k+1}):=\f(a_1)\F_{k}(\f,a_2,\ldots,a_{k+1}) - \F_{k}(\f,a_1a_2,\ldots,a_{k+1})-\ldots- \\
 \F_{k}(\f,a_2,\ldots,a_1a_{k+1})\,.
\end{multline}
\end{de}

Clearly $\F_{k}(\f,a_1,\ldots,a_{k})$ 
are multilinear  functions on $A$ with values in $B$. For small $k$, we have explicit formulas such as
%\begin{equation*}
%  \F_2(\f,a_1,a_2) =\f(a_1)\f(a_2)-\f(a_1a_2)\,,
%\end{equation*}
%\begin{multline*}
%  \F_3(\f,a_1,a_2,a_3) =\f(a_1)\f(a_2)\f(a_3)-\f(a_1)\f(a_2a_3)-\f(a_2)\f(a_1a_3)\\
%  -\f(a_3)\f(a_1a_2)+2\f(a_1a_2a_3)\,.
%\end{multline*}
\begin{align*}
  \F_2(\f,a_1,a_2)&=\f(a_1)\f(a_2)-\f(a_1a_2)\,, \\
  \F_3(\f,a_1,a_2,a_3)&=\f(a_1)\f(a_2)\f(a_3)-\f(a_1)\f(a_2a_3)-\f(a_2)\f(a_1a_3)\\
  &  \qquad\qquad\qquad\qquad\qquad -\f(a_3)\f(a_1a_2)+2\f(a_1a_2a_3)\,.
\end{align*}
We see that $\F_2(\f,a_1,a_2)$ and $\F_3(\f,a_1,a_2,a_3)$
are symmetric in    $a_1,a_2,\ldots $. This is the general case though the   definition by formula~\eqref{eq.deffrob} is not manifestly symmetric. 

\begin{lm}[Frobenius]
The functions $\F_{k}(\f,a_1,\ldots,a_{k})$ are symmetric in the arguments $a_1,\ldots,a_k\in A$ for all $k$\,.
\end{lm}
\begin{proof}
  Induction in $k$. 
\end{proof}

Hence, as symmetric multilinear functions, $\F_{k}(\f,a_1,\ldots,a_{k})$  are completely defined by their values on coinciding arguments. % $\F_{k}(\f,a,\ldots,a)$.

%$\F_2(\f,a_1,a_2)=\f(a_1)\f(a_2)-\f(a_1a_2)$, $\F_3(\f,a_1,a_2,a_3)=\f(a_1)\f(a_2)\f(a_3)-\f(a_1)\f(a_2a_3)-\f(a_2)\f(a_1a_3)-\f(a_3)\f(a_1a_2)+2\f(a_1a_2a_3)$.

\begin{lm}[Buchstaber and Rees]
%The functions $\F_{k}(\f,a_1,\ldots,a_{k})$ are symmetric in $a_1,\ldots,a_{k}\in A$. 
For coinciding arguments,
\begin{equation}
 \F_{k}(\f,a,\ldots,a)=k!\,\ps_k(\f,a)\,,
\end{equation}
where $\ps_k(\f,a)$ are as in~\eqref{eq.psi}.
\end{lm}
\begin{proof}
%We see explicitly that $\F_2(\f,a_1,a_2)$ and $\F_3(\f,a_1,a_2,a_3)$ are symmetric; for all $k\geq 2$, this follows by induction (directly from~\eqref{eq.deffrob}).
%By induction in $k$, we can establish the symmetricity of $\F_{k}(\f,a_1,\ldots,a_{k})$. 
%Therefore, $\F_{k}(\f,a_1,\ldots,a_{k})$ as symmetric multilinear functions are completely defined by their values on coinciding arguments $\F_{k}(\f,a,\ldots,a)$.
 Denote $\varphi_k(\f,a):=\F_{k}(\f,a,\ldots,a)$. 
Induction in $k$ 
using~\eqref{eq.deffrob}   gives the identity
\begin{equation} \label{eq.indukt}
    \Phi_k(a^m,a,\ldots,a)=\sum_{i=0}^{k-1}(-1)^i\frac{(k-1)!}{(k-i-1)!}\,f(a^{m+i})\,\varphi_{k-i-1}(\f,a)\,,
\end{equation}
for all $m$. We actually need $m=2$. Applying   formula~\eqref{eq.deffrob}  to $\varphi_{k+1}(\f,a)=\F_{k+1}(\f,a,\ldots,a)$, together
with~\eqref{eq.indukt} for $m=2$, gives a recursive relation  for the
functions $\varphi_k$:
\begin{equation*}
    \varphi_{k+1}(a)=\sum_{i=1}^{k+1}(-1)^{i-1}\frac{k!}{(k+1-i)!}\,\f(a^i)\,\varphi_{k+1-i}(a)\,,
\end{equation*}
for all $k=1,2,\ldots \ $. Denote provisionally
$\overline\si_k(a):=\frac{1}{k!}\,\varphi_k(\f,a)$.  We can rewrite this as a
recursive relation of the form
\begin{equation*}
    (k+1)\overline\si_{k+1}(a)=\sum_{i=1}^{k+1}(-1)^{i-1}\,\f(a^i)\,\overline\si_{k+1-i}(a)\,,
\end{equation*}
which is precisely  the   relation that defines   Newton's polynomials  (see~Macdonald~\cite[Ch. I]{macdon:symm}). Hence $\overline\si_k(a) = \ps_k(\f,a)$ and $\varphi_k(\f,a)=k!\ps_k(\f,a)$ as claimed.
\end{proof}
Given homogeneous $\f$-polynomials $\ps_k(\f,a)$,  the multilinear functions $\F_{k}(\f,a_1,\ldots,a_k)$ are obtained from them by polarization.  See   e.g. Eq. 3.3 in \cite{tv:shortproof}.

\subsection{How we obtain the main statements of Buchstaber and Rees}

\begin{de}[Buchstaber and Rees]
\label{def.nhom}
Fix $n=1,2,3,\ldots$
A linear map $\f\co A\to B$  is an algebra  \emph{$n$-homomorphism} if $\f(1)=n$ and $\F_{n+1}(\f,a_1,\ldots,a_{n+1})=0$ for all $a_i$\,. (Automatically, $\F_N=0$ for all $N\geq n+1$.)
\end{de}

\begin{thm}[\cite{tv:frobenius-umn-eng}]
\label{thm.first}
 A linear map $\f\co A\to B$  is a (Buchstaber--Rees)  $n$-homomorphism  if and only if $R(\f,a,z)$ is a polynomial of degree $n$.
\end{thm}
\begin{proof}
 Indeed, $\f$ is an $n$-homomorphism if and only if $\ps_k(\f,a)=0$ for all $k\geq n+1$. %(In greater detail, see in \cite{tv:shortproof}.)
\end{proof}

\begin{rem}
 In the  first work~\cite{buchstaber_rees:1996}, Buchstaber and Rees  defined $n$-homomorphisms using a function $p(a,t)$ of a variable $t$, which    by definition was  a monic polynomial of a fixed degree $n$; they abandoned it the next works in favour of a definition based on   Frobenius recursion. In  the hindsight,  $p(a,t)$ is similar to our characteristic function, but unlike $R(\f,a,z)$, it cannot be anything except for a   polynomial.
\end{rem}

\begin{rem}
Let $\f$ be an $n$-homomorphism in the sense of Definition~\ref{def.nhom}, then by Theorem~\ref{thm.first},
$R(\f,a,z)$ is a polynomial in $z$ of degree $n$. If we set $p(\f,a,t):=t^nR(\f,a,-t^{-1})$, the function $p(\f,a,t)$ will be a monic polynomial in $t$ also of   degree $n$. One can check that it is exactly the monic polynomial $p(a,t)$ of~\cite{buchstaber_rees:1996}. We see that to define $p(a,t)$, one needs to assume $n$ given and to presuppose  that $\f$ is an $n$-homomorphism. This is unlike   $R(\f,a,z)$, for   which we do not have to make any a priori assumptions  about algebraic properties of $\f$ and whose  functional properties relative to $z$ are meant to encode the former. Morally, the relation between $p(\f,a,t)$ and $R(\f,a,z)$ is similar to that between the characteristic polynomial $\det(\mathbf{A}-tE)$ of a linear operator in an (ordinary) vector space of dimension $n$ and the rational characteristic function~\cite{tv:ber} $\Ber(E+z\mathbf{A})$  natural to deal with for super vector spaces.
 %(The relation between $R(\f,a,z)$ and $p(\f,a,t)$ can be compared with that between $\det(1+\mathbf{A} z)$ and $\det (\mathbf{A}-t 1)$.)
%The polynomial $p(t)=p(\f,a,t)$ was used in the   first work of Buchstaber and Rees~\cite{buchstaber_rees:1996}.
\end{rem}

The following      important properties were originally found by Buchstaber and Rees. We can give for them a simple proof.
\begin{thm}
 The sum of an $n$-homomorphism and an $m$-homomorphism is an $(n+m)$-homomorphism.
 The composition of an $n$-homomorphism and an $m$-homomorphism is an $nm$-homomorphism.
\end{thm}
\begin{proof} (Following \cite{tv:frobenius-umn-eng, tv:frobbieloviezha, tv:shortproof}.)

Note that if $R(\f,a,z)$ is polynomial, then $\ch(\f)=\f(1)$ is its degree (as the order of the pole at infinity, see above). Consider the first statement, for $\f,\g\co A\to B$. We have $\f(1)=n$, $\g(1)=m$, and $R(\f,a,z)$ is a polynomial of degree $n$, $R(\g,a,z)$ is a polynomial of degree $m$. Then $(\f+\g)(1)=n+m$ and by \eqref{eq.expprop}, $R(\f+\g,a,z)=R(\f,a,z)R(\g,a,z)$ is a polynomial of degree $n+m$.
Consider the second statement, for $\f\co A\to B$ and $\g\co B\to C$. We have $(\g\circ \f)(1)=\g(\f(1))=\g(n1)=n\g(1)=nm$. Also,
\begin{equation}
 R(\g\circ \f,a,z)=e^{\g\f\ln(1+az)}=e^{\g\ln R(\f,a,z)}\,;
\end{equation}
but $e^{\g\ln (b)}=\Ber_{\g}(b)$ is the leading coefficient of the characteristic function of $\g$ at infinity and since $R(\g,b,z)$ is a polynomial of degree $m$, it is a homogeneous  $\g$-polynomial  of $b$ (see more details in~\cite[Proposition 4.2]{tv:shortproof}). Therefore, substituting $b=R(\f,a,z)$, which is a polynomial in $z$ of degree $\leq n$, gives a polynomial in $z$ of degree $\leq nm$.
\end{proof}

The crucial discovery of Buchstaber and Rees was an identification of their  $n$-homomorphisms $A\to B$ with the     algebra homomorphismsn (or ``$1$-homomorphisms'')  $S^n(A)\to B$, where $S^n(A)\subset A\otimes \ldots \otimes A$ is considered with the subalgebra structure.
\begin{thm}[Buchstaber and Rees]
An identification of the    $n$-homomorphisms $A\to B$ with the     homomorphisms $S^n(A)\to B$
is given as follows: if $\f\co A\to B$ is an $n$-homomorphism, then the formula
\begin{equation}
 F(a_1,\ldots,a_n):=\frac{1}{n!}\,\F_n(\f,a_1,\ldots,a_n)
\end{equation}
%(the top term of the Frobenius recursion)
defines a homomorphism $F\co S^n(A)\to B$; conversely, if $F\co S^n(A)\to B$ is an algebra homomorphism, then the linear map $\f\co A\to B$ defined by
\begin{equation}
 \f(a):=F\bigl(a\otimes 1\otimes\ldots\otimes 1+ \ldots + 1\otimes \ldots \otimes 1 \otimes a\bigr)
\end{equation}
is an $n$-homomorphism; and these two constructions are mutually inverse.
\end{thm}
\begin{proof}(Following~\cite{tv:frobenius-umn-eng}, with more details given in \cite{tv:shortproof}) Our key formula is the following:
\begin{equation}
 F\bigl(\det\bigl(1+\L(a)z\bigr)\bigr)=R(\f,a,z)
\end{equation}
where $\L(a)=\diag\bigl(a\otimes 1\otimes\ldots\otimes 1, \ldots, 1\otimes \ldots \otimes 1 \otimes a\bigr)\in \Mat(n, A^{\otimes n})$. The idea is that $F$ in the left-hand side determines $\f$ in the right-hand side, and conversely; and this gives the formulas of Buchstaber and Rees. Moreover, the desired properties of $F$ and $\f$ are established almost without effort. Roughly, everything is achieved by expanding both sides in $z$ and comparing the coefficients. In particular, the most difficult fact   that $F$ is a homomorphism provided that $\f$ is an $n$-homomorphism follows from the observation that elements of the form $a\otimes \ldots \otimes a$ span $S^n(A)$ and the multiplicativity of the $\f$-Berezinian~Eq \eqref{eq.multber}. Details are in \cite[\S5]{tv:shortproof}.\footnote{Note the correct attribution of these results as they were republished without references in a recent book of K.~W.~Johnson. The author admitted his fault. See \emph{Group Determinants and Representation Theory}, LNM 2233, Springer 2019. \\
 The proof of the Buchstaber--Rees theorem  on pp.276--277 was taken almost verbatim but without   reference from~\cite[p.1341--1342]{tv:shortproof}.
  The pieces at the bottom of p. 273 and on p. 274 were taken almost verbatim from~\cite[p. 1336]{tv:shortproof}.
   A paragraph on p. 275 was copied almost verbatim from the first paragraph of~\cite[p. 623]{tv:frobenius-umn-eng}. (Paper~\cite{tv:frobenius-umn-eng} does not appear in the bibliography, while~\cite{tv:shortproof}  is listed, but no borrowings  are indicated).}
\end{proof}

\section{Our generalization of the Buchstaber--Rees theory}

Here we return to a general rational function $R(\f,a,z)$.

\begin{de}[\cite{tv:frobenius-umn-eng}]
A linear map $\f\co A\to B$ is a \emph{$p|q$-homomorphism} if its characteristic function can be written as the ratio of polynomials of degrees $p$ and $q$.
\end{de}

(We assume that the fraction is irreducible.)

Then $\ch(\f)=\f(1)=p-q\in \ZZ$, which may now be negative.

\begin{ex} If $\f\co A\to B$ is an algebra homomorphism, then $R(\f,a,z)=1+\f(a)z$, hence for the negative of $\f$,
\begin{equation}
    R(-\f,a,z)=e^{-\f\ln(1+z)}=\frac{1}{e^{\f\ln(1+z)}}=\frac{1}{1+\f(a)z}\,.
\end{equation}
So $-\f$ is $0|1$-homomorphism. Also, $(-\f)(1)=-\f(1)=-1$.
\end{ex}

This immediately generalizes:

\begin{thm}[\cite{tv:frobenius-umn-eng, tv:frobbieloviezha, tv:shortproof}]\label{thm.pq}
If $\f$ is a $p$-homomorphism and $\g$ is a $q$-homomorphism, then $\f-\g$   is a $p|q$-homomorphism. In particular, if $\f_{\a}$ are homomorphisms and $n_{\a}\in \ZZ$, then
\begin{equation}
   \f:= \sum_{\a}n_{\a}\f_{\a}
\end{equation}
is a $p|q$-homomorphism where $p=\sum_{n_{\a}>0} n_{\a}$ and $q=-\sum_{n_{\a}<0} n_{\a}$, and $\ch(\f)=\sum n_{\a}$.
\end{thm}
\begin{proof} By the exponential property Eq~\eqref{eq.expprop}.
\end{proof}

The crucial geometric interpretation of the Buchstaber--Rees result was a generalization of the Gelfand--Kolmogorov theorem~\cite{gelfand-kolmogorov} that identifies a Hausdorff compact topological space $X$ with an ``affine algebraic variety'' in the infinite-dimensional space $C(X)^*$, the algebraic dual of the algebra of real-valued continuous functions, specified by the  system of  quadric equations on $\f\in C(X)^*$\,:
\begin{equation}
\f(a^2)=\f(a)^2\,,
\end{equation}
for all $a\in C(X)$. The corresponding  Buchstaber--Rees statement identifies the $n$-th symmetric power $\Sym^n(X)$ with the ``affine algebraic variety''  in $C(X)^*$ specified by the system of polynomial equations of degree $n+1$ on $\f\in C(X)^*$\,:
\begin{equation}
\psi_{n+1}(\f,a)=0
\end{equation}
for all $a\in C(X)$, which expresses the condition that $\f\in C(X)^*$ is an $n$-homomorphism.

\begin{rem}
 In the Gelfand--Kolmogorov theorem, it is essential that real functions are considered and no topology is used in the algebra   $C(X)$. (Its algebraic structure alone is enough to recover both the space $X$ as a set and the topology of $X$.) In both regards it differs from the theory of normed rings (Banach algebras) developed by Gelfand. It was said that Kolmogorov had an aversion to complex numbers.
\end{rem}

A further generalization, from homomorphisms of Gelfand--Kolmogorov and $n$-homomorphisms of Buchstaber--Rees, is based on the following new  geometric notion.

\begin{de}[\cite{tv:frobenius-umn-eng}]
The \emph{$p|q$-th (generalized) symmetric power} of a space $X$, notation: $\Sym^{p|q}(X)$, is defined as the quotient of $X^{p+q}$ with respect to the action of the group $S_p\times S_q$
and the additional relation
\begin{equation}
    (x_1,\ldots,x_{p-1},y,x_{p+1},\ldots,x_{p+q-1},y)\sim (x_1,\ldots,x_{p-1},z,x_{p+1},\ldots,x_{p+q-1},z)\,.
\end{equation}
\end{de}

\begin{thm}[\cite{tv:frobenius-umn-eng, tv:frobbieloviezha, tv:shortproof}]
Every point $\mathbf{x}=[x_1,\ldots,x_{p+q}]\in \Sym^{p|q}(X)$ defines a $p|q$-homomorphism of ``evaluation'' $\ev_{\mathbf{x}}\co C(X)\to \RR$,
\begin{equation}
    \ev_{\mathbf{x}}(a):=a(x_1)+\ldots+a(x_p)-a(x_{p+1})-\ldots-a(x_{p+q})\,,
\end{equation}
for $a\in C(X)$. The  image of $\Sym^{p|q}(X)$ in $C(X)^*$ satisfies the algebraic equations
\begin{equation}
\f(1)=p-q \quad \text{and} \quad
\begin{vmatrix}
      \ps_k(\f,a)  & \dots & \ps_{k+q}(\f,a)  \\
      \dots & \dots & \dots \\
      \ps_{k+q}(\f,a)  & \dots & \ps_{k+2q}(\f,a)  \\
    \end{vmatrix}=0
\end{equation}
for all $k\geq p-q+1$ and all $a\in A$\,.
\end{thm}
\begin{proof} The first statement follows from the definition of $\Sym^{p|q}(X)$ and Theorem~\ref{thm.pq}. The second statement follows from the characterization of expansion of a rational function (see e.g.~\cite{tv:ber}).
\end{proof}

Ideally, we would like to be able to identify $\Sym^{p|q}(X)$ with this ``algebraic variety''. But this has not been proved so far (that all its points come from the points of $\Sym^{p|q}(X)$).

%\bibliographystyle{plain}
%\bibliography{/home/ted_voronov/Desktop/Work/Local_TeX_Files/bibtex/bib/misc/geometry} %% Office UNIX
%\bibliography{E:/Ted/Documents/Научное/Bibliography/geometry.bib} %% Laptop Windows
%\bibliography{geometry} %% Laptop Windows
%\end{document}

\def\cprime{$'$}

\end{document}